\documentclass[a4paper,11pt]{llncs}
\usepackage[utf8]{inputenc}
\usepackage[T1]{fontenc}
\usepackage{zi4}
\usepackage[a4paper,margin=2.5cm]{geometry}

\usepackage{cite}
\usepackage{mathtools,amsfonts}
\usepackage{mathptmx}
\usepackage[english]{babel}
\usepackage{enumerate}
\usepackage{amsmath}
\usepackage{amssymb,latexsym}
\usepackage{varioref}
\usepackage{hyperref}
\usepackage{cleveref}
\usepackage{xcolor}
\usepackage{xspace}
\usepackage{csquotes}
\usepackage[super]{nth}
\usepackage{url}

\hypersetup{
	linktoc = page,
	pdfpagemode = UseNone,
	colorlinks,
	linkcolor={red!50!black},
	citecolor={blue!50!black},
	urlcolor={blue!80!black}
}

\usepackage{tikz}

\usepackage{booktabs}
\newcommand{\dig}[1]{\mbox{\texttt{#1}}}
\newcommand{\rl}{\mathsf{RunL}}
\newcommand{\lsa}{\mathsf{LSA}}
\newcommand{\lsb}{\mathsf{LSB}}

\renewenvironment*{example}[0]{\refstepcounter{example}\smallskip\noindent\textbf{Example~\theexample.\xspace}}{}
\renewenvironment*{conjecture}[0]{\refstepcounter{conjecture}\smallskip\noindent\textbf{Conjecture~\theconjecture.\xspace}}{}

\DeclareFontFamily{U}{mathb}{\hyphenchar\font45}
\DeclareFontShape{U}{mathb}{m}{n}{
      <5> <6> <7> <8> <9> <10> gen * mathb
      <10.95> mathb10 <12> <14.4> <17.28> <20.74> <24.88> matha12
      }{}
\DeclareSymbolFont{mathb}{U}{mathb}{m}{n}

\DeclareMathSymbol{\downuparrows}{3}{mathb}{"D7}
\DeclareMathSymbol{\bft}{3}{mathb}{"FD}

\begin{document}

\title{The Look-and-Say The Biggest Sequence Eventually Cycles}

\author{%
    \'Eric Brier\inst{2} \and Rémi Géraud-Stewart\inst{1} \and David Naccache\inst{1}\and Alessandro Pacco\inst{1} \and Emanuele Troiani\inst{1}
} 

\institute{
    ÉNS (DI), Information Security Group,
CNRS, PSL Research University, 75005, Paris, France.\\
45 rue d'Ulm, 75230, Paris \textsc{cedex} 05, France\\
\email{\url{remi.geraud@ens.fr}},~~~~  
\email{\url{given\_name.family\_name@ens.fr}}
	\and
	Ingenico Laboratories, 75015, Paris, France.\\
	\email{\url{eric.brier@ingenico.com}}\\
}
\maketitle

\begin{abstract}
In this paper we consider a variant of Conway's sequence (\textsf{OEIS A005150, A006715}) defined as follows: the next term in the sequence is obtained by considering contiguous runs of digits, and rewriting them as $ab$ where $b$ is the digit and $a$ is the \emph{maximum} of $b$ and the run's length. We dub this the \enquote{look-and-say the biggest} ($\lsb$) sequence.

Conway's sequence is very similar ($b$ is just the run's length). For any starting value except \dig{22}, Conway's sequence grows exponentially: the ration of lengths converges to a known constant $\lambda$. We show that $\lsb$ does not: for every starting value, $\lsb$ eventually reaches a cycle. Furthermore, all cycles have a period of at most 9.
\end{abstract}

\section{Introduction}

The look-and-say (\textsf{LS}) sequence \cite{conway2012book}, also known as the \emph{Morris} or the \emph{Conway sequence} \cite{conway1987weird,hilgemeier1996one,ekhad1997proof} is a recreational integer sequence having very intriguing properties. A \textsf{LS} sequence is obtained iteratively by reading off the digits of the current value, and counting the number of digits in groups of the identical digit.\smallskip

Following Conway's work, numerous variants of \textsf{LS} sequences were proposed and studied, e.g., Pea Pattern sequences \cite{kowacs2017studies,pea}, Sloane's sequences\cite{slo} or Oldenburger--Kolakoski sequences \cite{kola,kola2,lagarias1992number} (\textsf{OEIS A000002}).\smallskip

In this paper we consider a variant of $\textsf{LS}$ sequence: the next term in the sequence is obtained by considering contiguous runs of digits, and rewriting them as $ab$ where $b$ is the digit and $a$ is the \emph{maximum} of $b$ and the run's length (this is formally defined further below). We dub this the \enquote{look-and-say the biggest} ($\lsb$) sequence. An $\lsb$ sequence is therefore very similar to a \textsf{LS} sequence; however it behaves very differently.

Indeed, Conway showed \cite{conway1987weird} that for all starting values (except \dig{22}), \textsf{LS} sequences produce ever-growing numbers, with a growth ratio converging to some constant $\lambda$ (which is the root of an explicit degree-$71$ polynomial). As we will show, however, for \emph{every} starting value, $\lsb$ sequences eventually reach a cycle. When they do, this cycle repeats itself in at most 9 iterations (\Cref{thm:sam}).

\section{Notations and definitions}
In this paper we assume that numbers are written in base 10. Any integer $T$ can thus be written $T = t_1t_2 \cdots t_k$ with $t_1, \dotsc, t_k \in \{{\dig{0}}, {\dig{1}}, \dotsc, {\dig{9}}\}$. To avoid any ambiguity, $ab$ will denote the \emph{concatenation} of the numbers $a$ and $b$; accordingly $a^b$ indicates that a digit $a$ is repeated $b$ times. If we want to emphasise concatenation we use $a\|b$ instead of $ab$.

\begin{definition}[Run-length representation]
Let $T\in \mathbb{N}^*$, we can write 
\begin{equation*}
    T = \underbrace{a_1\ldots a_1}_{n_1}\underbrace{a_2\ldots a_2}_{n_2}\ldots\underbrace{a_k\ldots a_k}_{n_k}
\end{equation*}
with $a_1\neq a_2, a_2\neq a_3, \dotsc, a_{k-1}\neq a_{k}$. The \emph{run-length representation} of $T$ is the sequence $\rl(T)= a_1^{n_1}a_2^{n_2} \cdots a_k^{n_k}$. Conversely, any finite sequence of couples $(a_i, b_i)_i$ where $a \in \mathbb N^*$ and ${\dig{0}}\leq b_i\leq {\dig{9}}$ is such that $b_{i-1}\neq b_i\neq b_{i+1}$, corresponds to an integer with run-length representation $(a_i^{b_i})_i$. 
\end{definition}
Note that the run-length representation of an integer is unique. 

\begin{definition}[Pieces]
If $N = (a_i^{b_i})$ is a run-length encoded integer, we call each $a_i^{b_i}$ a \emph{piece} of $N$.
\end{definition}

\begin{definition}[Look-and-say-again sequence]
Let $T_0$ be a decimal digit, and for each $T_n$ define
\begin{equation*}
    T_{n+1} = n_1n_1a_1a_1n_2n_2a_2a_2\cdots n_kn_ka_ka_k.
\end{equation*}
where $(a_i^{n_i})_i = \rl(T_n)$. We call the sequence $(T_k)_{k \in \mathbb N}$ the \emph{look-and-say-again sequence} of seed $T_0$, and denote it by $\lsa(T_0)$.
\end{definition}

\begin{definition}[Look-and-say-the-biggest sequence]
Let $S_0\in\mathbb{N}^*$, and for every integer $S_n\in\mathbb{N}^*$, $(a_i^{n_i})_i = \rl(S_n)$, let
\begin{equation*}
    S_{n+1} = \max(n_1,a_1)a_1\max(n_2,a_2)a_2 \cdots \max(n_k,a_k)a_k.
\end{equation*}
We call $(S_n)_{n \in \mathbb N}$ the \emph{look-and-say-the-biggest} sequence of seed $S_0$, and denote it by $\lsb(S_0)$.
\end{definition}

\begin{example}
$\lsb(1) = 
\small{    
    \dig{1}\to\dig{11}\to\dig{21}\to\dig{2211}\to\dig{2221}\to\dig{3211}\to\dig{332211}\to\dig{332221}\leftrightarrows\dig{333211}}$
\end{example}

\section{The look-and-say-the-biggest sequence}

\begin{theorem}[Eventual periodicity]\label{thm:sam}
Let $s$ be an integer, then the sequence $\lsb(s)$ is eventually periodic, of period $\tau \leq 9$.
\end{theorem}
To prove this result we introduce several useful definitions and lemmas.

\begin{definition}[Maxmaps]
Let $a_i^{n_i}$ be a piece, we define the \emph{partial maxmap}:
\begin{equation*}
    z: a_i^{n_i} \mapsto \max(a_i, n_i) a_i
\end{equation*}
which sends a piece to an integer. We extend this definition to work on integers: if $(a_i^{n_i})_i = \rl(N)$, we define the \emph{maxmap}:
\begin{equation*}
    Z: N \mapsto z(a_1^{n_1})z(a_2^{n_2}) \cdots z(a_k^{n_k})
\end{equation*}
which maps integers to integers. Finally, for every integer $m > 0$, we write $Z^m = Z \circ Z \circ \cdots \circ Z$, 
where $Z$ appears $m$ times in the composition.

\end{definition}
\begin{example} 
$Z(\dig{11193222}) = \dig{31993332}$.
\end{example}

\begin{remark}
With these notions in place, \Cref{thm:sam} can be reformulated as follows: for any integer $s$,
if $s$ does not contain more than 9 continuous identical digits, then $s \in \{Z^m(s) \mid 1 \leq m \leq 8\}$.
\end{remark}

\begin{remark}
Assume that $s$ contains more than $9$ consecutive identical digits $a$, but fewer than $100$; write this number $n$. Then the corresponding piece in the run-length representation is $a^n$, which will be mapped by $z$ to $na$ (as $n > a$ by assumption). But $n$ itself is made of two digits so that $na$ contains at most 3 consecutive identical digits. By iteratively applying $Z$ to a number having more than 9 consecutive digits, we eventually obtain a number with at most 9 consecutive digits.
\Cref{thm:sam} will therefore follow from this observation and the following result:
\end{remark}

\begin{lemma}\label{lem:periodweak}
Let $s \in \mathbb N$, $(a_i^{n_i})_i = \rl(s)$. If $n_i \leq 9$ for all $i$ then the sequence $Z^n(s)$ periodic, of period $\tau < 9$.
\end{lemma}

\begin{definition}[Kid and Adult Types]
We split digits into two sets: digits that are smaller or equal to $3$ are \emph{kid type}, the others are \emph{adult type}. Pieces $a^n$ where $a$ is a kid (resp. adult) are of type kid (resp. adult). Numbers made \emph{only} of type kid (resp. adult) digits are of type kid (resp. adult).
\end{definition}

\begin{lemma}\label{lem:ZtypeB}
$Z$ maps adult type numbers to adult type numbers.
 \end{lemma} 
 \begin{proof}[of \Cref{lem:ZtypeB}]
 If $a^n$ is an adult piece, $z(a^n) = \max(n, a)a$ is an adult number, since $a$ is an adult digit by hypothesis and $\max(n, a)$ is adult if either numbers is. The result on $Z$ follows.
\qed
\end{proof}

\begin{lemma}[Fixed points of $Z$, sufficient condition]\label{lem:fixpt}
Let $s$ be an integer, $(a_i^{n_i})_i = \rl(s)$. If $n_i = 2, a_i \geq \dig{2}$ for all $i$ then $Z(s) = s$.
\end{lemma}
\begin{proof}[of \Cref{lem:fixpt}]
We have, for all $i$, $z(a_i^{n_i}) = \max(n_i, a_i)a_i = a_ia_i$. The result on $Z$ follows.
\qed
\end{proof}

\begin{theorem}[Adult seeds give constant sequences]\label{thm:typeB}
Let $s \in \mathbb N$ be adult, and assume that $s$ does not have more than 9 consecutive identical digits, then $Z^3(s) = Z^2(s)$. In particular, the sequence $\lsb(s)$ is constant from the third term on.
\end{theorem}
\begin{proof}[of \Cref{thm:typeB}]
Let $(a_i^{n_i})_i = \rl(s)$. By hypothesis, $a_i$ are adult and $n_i \leq 9$. 
We have
\begin{equation*}
    z(a_i^{n_i}) = 
    \begin{cases}
    a_ia_i & \text{if $n_i \leq a_i$} \\
    n_ia_i & \text{otherwise}
    \end{cases}
\end{equation*}
In the first case, $z(a_i^{n_i})$ is an adult number of length 2, and in the second case $n_i > a_i$ so that we again have an adult of length 2. Applying $Z$ to $s$, it is possible that two consecutive digits (coming from different pieces of $s$) are identical, but no more: as a result, if we write $(b_i^{m_i})_i =  \rl(Z(s))$, then $m_i \leq 3$ and $b_i$ is adult.
Therefore, $Z^2(s)$ satisfies the conditions of \Cref{lem:fixpt}, yielding $Z^3(s) = Z^2(s)$.
\qed
\end{proof}

\begin{remark}
\Cref{thm:typeB} proves \Cref{lem:periodweak}, and therefore \Cref{thm:sam}, for adult integers. We now turn our attention to the remaining case of kid integers.
\end{remark}

\begin{lemma}\label{lem:fixptA}
Let $s \in \mathbb N$ of kid type and assume that no digit appears more than 3 times consecutively in $s$. Then $Z(s)$ is kid and no digit appears more than 3 times consecutively in it.
\end{lemma}

\begin{proof}[of \Cref{lem:fixptA}]
Let $(n_i, a_i)_i = \rl(s)$, we can exhaust all possibilities:
\begin{align*}
      \dig{0}^1 & \xrightarrow{z} \dig{10} & \dig{0}^2 & \xrightarrow{z} \dig{20}
    & \dig{0}^3 & \xrightarrow{z} \dig{30}
    & \dig{1}^1 & \xrightarrow{z} \dig{11}
    & \dig{1}^2 & \xrightarrow{z} \dig{21}
    & \dig{1}^3 & \xrightarrow{z} \dig{31} \\
      \dig{2}^1 & \xrightarrow{z} \dig{22}
    & \dig{2}^2 & \xrightarrow{z} \dig{22}
    & \dig{2}^3 & \xrightarrow{z} \dig{32}
    & \dig{3}^1 & \xrightarrow{z} \dig{33}
    & \dig{3}^2 & \xrightarrow{z} \dig{33}
    & \dig{3}^3 & \xrightarrow{z} \dig{33}
\end{align*}
Clearly, $z(a_i^{n_i})$ is kid and of length 2. Furthermore, for $Z$ to contain more than 3 consecutive identical digits, $\rl(s)$ would need to contain a sub-sequence of the form $ \beta^{\alpha}\beta^{\alpha'}$ with $\alpha, \alpha'\in \{1,2,3\}$ and $ \beta \in \{\dig{1},\dig{2},\dig{3}\}$. This is impossible by definition of the run-length representation. \qed 
\end{proof}

\begin{lemma}[Last digit]\label{lem:lastdigit}
Let $s$ be an integer, then the last digit of $Z(s)$ is the same as the last digit of $s$.
\end{lemma}
\begin{proof}
The proof is immediate from the definition of $Z$. \qed 
\end{proof}

\begin{theorem}[Small kid seeds]\label{thm:smalltypeA}
Consider again a kid integer $s$, with no more than three consecutive identical digits, then the sequence $Z^n(s)$ loops after at most $7$ iterations, i.e. $s \in \{Z^m(s) \mid 1 \leq m \leq 8\}$.
\end{theorem}

\begin{proof}
\Cref{fig:smallnumbers} shows how $Z$ maps each piece of the integers satisfying the theorem's conditions. We notice that a fixed point is reached in one step for $\dig{2}$, $\dig{22}$, $\dig{3}$, $\dig{33}$, $\dig{333}$, and the only end nodes that are not fixed points are $\dig{10}$, $\dig{20}$, $\dig{30}$, $\dig{21}$, $\dig{32}$, $\dig{31}$ --- these are the ones we shall focus on.
We list the successive mappings of these three nodes up to the point before they enter a loop in \Cref{tab:fossils}, and refer to the list of integers in that table as \emph{fossil integers}.

\begin{figure}[!th]
    \centering 
    \begin{tikzpicture}
\node (1) at (0, 0) {\phantom{\dig{00}}\dig{1}};
\node (11) at (1.5, 0) {\dig{11}};
\node (21) at (3, 0) {\dig{21}};
\node (111) at (0, -1) {\dig{111}};
\node (31) at (3, -1) {\dig{31}};
\node (222) at (0, -2) {\dig{222}};
\node (0) at (0, -3) {\phantom{\dig{00}}\dig{0}};
\node (10) at (3, -3) {\dig{10}};
\node (00) at (5, -3) {\dig{00}};
\node (20) at (8, -3) {\phantom{\dig{0}}\dig{20}};
\node (000) at (2.5, -4) {\dig{000}};
\node (30) at (5.5, -4) {\dig{30}};
\node (32) at (3, -2) {\dig{32}};
\node (2) at (5, 0) {\phantom{\dig{0}}\dig{2}};
\node (22) at (7, 0) {\dig{22}};
\node (3) at (5, -2) {\phantom{\dig{0}}\dig{3}};
\node (33) at (6.5, -2) {\dig{33}};
\node (333) at (8, -2) {\dig{333}};
\draw[>=latex,->] (0) to node[above,midway,scale=0.8]{$Z$} (10);
\draw[>=latex,->] (00) to node[above,midway,scale=0.8]{$Z$} (20);
\draw[>=latex,->] (000) to node[above,midway,scale=0.8]{$Z$} (30);
\draw[>=latex,->] (1) to node[above,midway,scale=0.8]{$Z$} (11);
\draw[>=latex,->] (11) to node[above,midway,scale=0.8]{$Z$} (21);
\draw[>=latex,->] (111) to node[above,midway,scale=0.8]{$Z$} (31);
\draw[>=latex,->] (222) to node[above,midway,scale=0.8]{$Z$} (32);
\draw[>=latex,->] (2) to node[above,midway,scale=0.8]{$Z$} (22);
\draw[] (22) to[in=90,out=25] node[pos=1,right,scale=0.8]{$Z$} (8.5, 0);
\draw[>=latex,->] (8.5, 0) to[out=-90,in=-25] (22);
\draw[>=latex,->] (3) to node[above,midway,scale=0.8]{$Z$} (33);
\draw[>=latex,->] (333) to node[above,midway,scale=0.8]{$Z$} (33);
\draw[] (33) to[out=120,in=-180] node[pos=2,right,midway,scale=0.8]{$~Z$} (6.5, -1);
\draw[>=latex,->] (6.5, -1) to[out=0,in=60] (33);
\end{tikzpicture}
    \caption{Effect of $Z$ on small kid integers.} \label{fig:smallnumbers}
\end{figure}
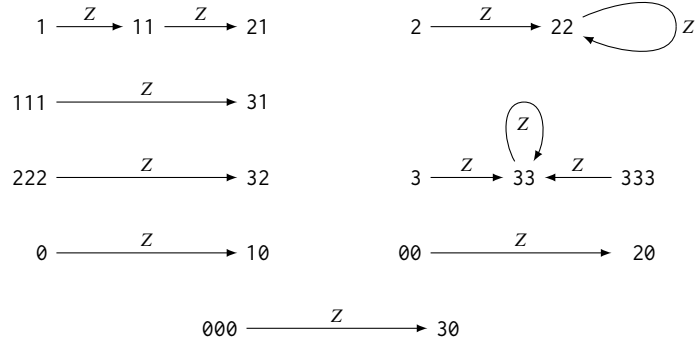
\begin{table}[!th]
\centering 
\caption{Cycles from small kid seeds. We call the numbers appearing here \emph{fossils}.}\label{tab:fossils}
\begin{tabular}{cc cc cc cc cc cc }
&&&&&&&&&& Steps & Period\\\toprule 
$\dig{32}$&$\to$&$\dig{3322}$&\hspace{-0.3cm}$\bft$&&&&&&&1&1\\\midrule
$\dig{31}$&$\to$&\hspace{-0.3cm}$\dig{3311}$&\hspace{-0.7cm}$\to$&\hspace{-1.2cm}$\dig{3321}$&$\hspace{-1.7cm}\to$&\hspace{-2.2cm}$\dig{332211}$ &&&&&\\
&&&&&&\hspace{-2.2cm}$\downarrow$&&&&&\\
&&&&&&\hspace{-2.2cm}$\dig{332221}$&$\hspace{-2.6cm}\leftrightarrows$&$\hspace{-1.4cm}\dig{333211}$&&5&2\\
&&&&&&\hspace{-2.2cm}$\uparrow$&&&&&\\
$\dig{21}$&$\to$&\hspace{-0.3cm}$\dig{2211}$&$\hspace{-0.7cm}\to$&\hspace{-1.2cm}$\dig{2221}$&$\hspace{-1.7cm}\to$&\hspace{-2.2cm}$\dig{3211}$   &&&&& \\\midrule
$\dig{00}$&&$\dig{221110}$&$\to$&\hspace{-0.16cm}$\dig{223110}$&&$\dig{2233322110}$&&&&&\\
$\downarrow$&&$\uparrow$&&$\downarrow$&&$\downuparrows$&&&&7&2\\
$\dig{20}$&$\to$&$\dig{2210}$&&\hspace{-0.16cm}$\dig{22332110}$&$\to$&$\dig{2233222110}$&&&&&\\\midrule
$\dig{0}$    &       & $\dig{3110}$ & $\to$ & \hspace{-0.16cm}$\dig{332110}$ &&&&&&& \\ 
$\downarrow$&&$\uparrow$&&$\downarrow$&&&&&&6&2\\
$\dig{10}$   & $\to$ & $\dig{1110}$ &       & \hspace{-0.16cm}$\dig{33222110}$ & $\leftrightarrows$ & \hspace{-0.3cm}$\dig{33322110}$ &&&&& \\\midrule 
$\dig{000}$&       & $\dig{331110}$ & $\to$& \hspace{-0.16cm}$\dig{333110}$ &       &\hspace{-0.3cm}$\dig{33222110}$\\
$\downarrow$&&$\uparrow$&&$\downarrow$&&$\downuparrows$&&&&7&2\\
$\dig{30}$ & $\to$ & $\dig{3310}$   &      &  \hspace{-0.16cm}$\dig{332110}$ & $\to$&\hspace{-0.3cm}$\dig{33222110}$ 
\\
\bottomrule 
\end{tabular}
\end{table}

Assume that $s = s_0 \| p \| s_1$, with $s_0, p, s_1$ integers such that $Z(s_i) = s_i$, and such that $s_0$ does not end with $p$'s first digit and $s_1$ does not start with $p$'s last digit. If $p$ is not fossil then $Z(p) = p$ and therefore $Z(s) = s$. Otherwise, 
\begin{itemize}
    \item For $Z(p)$ to start with $\dig{3}$, we must have $p$ beginning with $\dig{3}b$, $\dig{33}b$, with $b \neq\dig{3}$, or three consecutive digits $aaa$. Then $Z(p)$ starts with $\dig{3}a$ and $Z^2(p)$ starts with $\dig{33}$ which is not fossil. By \Cref{lem:lastdigit}, the first digit of $Z(s_1)$ and the last digit of $Z(p)$ differ. As a result after two iterations of $Z$ we are in a situation $s' = s_0'\|p'\|s_1$ where 
    $s_0' = s_0\|\dig{33}$ and $p'$ does not begin with $\dig{3}$ nor three consecutive digits.
    \item For $Z(p)$ to start with $\dig{2}$, we must have $p$ beginning with $\dig{2}b$ with $b < \dig{2}$, or two consecutive digits $aa$ with $a \leq \dig{2}$. Then $Z(p)$ starts with $\dig{22}$ or $\dig{2}a$ respectively, and $Z^2(p)$ starts with $\dig{22}$. As above, after two iterations of $Z$ we are in a situation $s' = s_0'\|p'\|s_1$ where 
    $s_0' = s_0\|\dig{22}$ and $p'$ does not begin with $\dig{2}$ nor two consecutive digits.
    \item For $Z(p)$ to start with $\dig{1}$, we must have $p$ beginning with $ab$ with $a \leq \dig{1}$ and $b > \dig{1}$. Then $Z(p)$ starts with $\dig{10}$ or $\dig{11}$, and $Z^2(p)$ starts with $\dig{11}$ or $\dig{21}$ respectively. Then $Z^4(p)$ or $Z^3(p)$ respectively starts with $\dig{22}$. Thus after three or four iterations of $Z$ we are in a situation $s' = s_0'\|p'\|s_1$ where 
    $s_0' = s_0\|\dig{22}$ and $p'$ does not begin with an isolated digit.
\end{itemize}
Since $Z(p)$ cannot start with $\dig{0}$, these are the only cases. Therefore, after at most $4 + 2 + 2 = 8$ iterations there remain no possibility for $p$. In other terms $Z^m(s) = s$ for $m \leq 8$.

\qed
\end{proof}
We are now ready to complete the proof:
\begin{proof}[of \cref{lem:periodweak,thm:sam}]
We have proven the result for adult seeds (\Cref{thm:typeB}) and for certain kid seeds (\Cref{thm:smalltypeA}). 

Let $s$ be an integer, $(a_i^{n_i})_i = \rl(s)$. We partition the pieces $(a_i^{n_i})_i$ into kid type and adult types. The effect of $Z$ on kid pieces is the one that needs attention: either $n_i \leq 3$, or $n_i > 3$. In the second case, the effect of $Z$ is to introduce the adult digit $n_i$. Consider the type of pieces $(i-1)$ and $(i+1)$:
\begin{itemize}
    \item If piece $(i-1)$ is adult: then $n_i$ will fuse with the result of $z(a_{i-1}^{n_{i-1}})$ into a adult piece;
    \item Otherwise $n_i$ will be its own new adult piece, inserted between the result of $z$ on piece $(i-1)$ and the $a_i$. 
\end{itemize}
As a result, after applying $Z$ to $s$, we can partition $Z(s)$ as a concatenation of kid integers having no more than 3 consecutive identical digits, and adult integers. These correspond in particular to disjoint run-length representations, so that $Z$ acts on them independently.

We have shown that the action of $Z$ on such integers preserves their properties, and that $Z$ reaches a cycle in at most 7 iterations. Accounting for the extra application of $Z$ at the beginning of this proof, we have the claimed result.
\qed 
\end{proof}

\begin{remark}
\Cref{thm:sam} shows that we eventually reach a cycle; \Cref{lem:periodweak} shows that for certain integers (which do not repeat any digit more than 9 times in a row) this cycle is reached quickly. Nevertheless, some integers that \emph{do not} satisfy the conditions of \Cref{lem:periodweak} can still reach a cycle quickly: for instance, $\dig{2}^{22} \to \dig{222} \to \dig{3322}$ (two iterations). 
This brings the following question: what is the smallest integer $y$ so that $\lsb(y)$ takes \emph{more} than $8$ iterations to reach a cycle?  
\end{remark}

\begin{conjecture}\label{conj}
The smallest integer $y$ so that $\lsb(y)$ takes $9$ iterations to reach a cycle is $y = \dig{2}^x$ with $x = \dig{3}^{11}$. If we now (temporarily) use the standard multiplicative notation $a^b = \underbrace{a \times a \times \cdots \times a}_\text{$b$ times}$, then
\begin{equation*}
    y = \frac{2}{9}\left(10^{\left(10^{11}-1\right)/3}-1\right) \simeq 10^{10^{10.52}}.
\end{equation*}
\end{conjecture}
The convergence process in 9 steps is:

\begin{center}
{\small
\begin{tabular}{r}
$y=\dig{2}^x 
    \to x\dig{2} \to \dig{11312} \to \dig{21331122} \to \dig{2211332122} \to \dig{222133221122} 
    \to \dig{321133222122} \to \dig{33222133321122}$\\
    $\downuparrows$~~~~~~~~~~~~~\\
    $\dig{33321133222122}$
\end{tabular}}
\end{center}

\section{Open questions}
Although every $\lsb$ sequence cycles, it remains an open question to fully enumerate and classify the limit cycles. It also remains open to extend Conjecture~\ref{conj} and determine for every $n > 0$ the values $\sigma_n$ defined as the smallest positive integer such that $\lsb(\sigma_n)$ reaches a cycle in exactly $n$ iterations.

\bibliographystyle{alpha}
\bibliography{conway.bib}
\end{document}